\documentclass[12pt]{article}
\usepackage{amsmath,amssymb,amsfonts,amsthm,enumerate,url,hyperref}
\usepackage{eucal} 
\usepackage{xcolor}
\usepackage{youngtab}
\usepackage{young}
\usepackage{lscape}
\usepackage{tikz}
\usepackage{environ}
\usepackage{caption}
\usepackage{subcaption}
\makeatletter
\newsavebox{\measure@tikzpicture}
\NewEnviron{scaletikzpicturetowidth}[1]{%
  \def\tikz@width{#1}%
  \def\tikzscale{1}\begin{lrbox}{\measure@tikzpicture}%
  \BODY
  \end{lrbox}%
  \pgfmathparse{#1/\wd\measure@tikzpicture}%
  \edef\tikzscale{\pgfmathresult}%
  \BODY
}
\makeatother

\newcommand{\javaplex}{{\texttt{Javaplex}} }

\newcommand{\lmp}{p}
\newcommand{\lm}{L}

\newcommand{\R}{\mathbb{R}}

\newcommand{\eps}{\epsilon}

\newcommand{\sig}{\sigma}

\newcommand{\vor}{V}
\newcommand{\ordering}[2]{{#1}({#2})}

\newtheorem{thm}{Theorem}

\newtheorem{lemma}{Lemma}

\theoremstyle{definition}
\newtheorem{defn}{Definition}

\newtheorem{remark}{Remark}

\title{A new construction for sublevel set persistence}
\author{Erik Carlsson, John Carlsson}
\begin{document}
\maketitle

\begin{abstract}
We construct a filtered simplicial complex $(X_\lm,f_\lm)$
associated to a subset $X\subset \R^d$, a function $f:X\rightarrow \R$ with compactly supported sublevel sets, and a collection
of landmark points $\lm\subset \R^d$.
The persistence values $f_\lm(\Delta)$ are defined
as the minimizing values of a family of
constrained optimization problems, whose domains are certain
higher order Voronoi cells associated to $\lm$.
We prove that $H_k^{a,b}(X_\lm)\cong H^{a,b}_k(X)$ in the case when $X=\R^d$
and $f$ is smooth, the
landmarks are sufficiently dense, and $a<b$ are generic.
We show that the construction produces
desirable results in some examples.
\end{abstract}
\section{Introduction}

Let $f:X\rightarrow \R$, and let
$X(a)=f^{-1}(-\infty,a]$ be the sublevel set with the induced topology.
The $k$th persistent homology group $H^{a,b}_k(X)$ is
defined as the image of the homomorphism
\begin{equation}
\label{eq:sublevcontin}
    i^{a,b}_k: H_k(X(a))\rightarrow H_k(X(b))
\end{equation}
which is induced from the inclusion map $i^{a,b}:X(a)\rightarrow X(b)$,
defined for $a\leq b$.

The homology groups in \eqref{eq:sublevcontin}
have several advantages over distance-based
persistent homology, but are generally difficult to compute,
particularly in higher dimensions.
One application of sublevel set persistence is
to topological data analysis, which involves 
computing the superlevel set persistence of
a suitable density estimator $\rho$
(or the sublevel set persistence of 
$-\log(\rho)$).
In a seminal paper \cite{chazal2013clustering}, 
the authors applied zeroth dimensional 
superlevel set persistence to define a novel 
clustering scheme, which
provably computed the correct number of clusters,
in a mathematically precise sense.
Another area where sublevel set persistence has been
well-suited is time series analysis, see for instance
\cite{chung2020time,chung2021persistent,myers2020separating,harer2013sliding,ravishanker2019topological,seversky2016pattern}.
The higher dimensional case has been studied using the Morse-Smale complex 
\cite{gunther2012efficient} in the context of
grayscale images, which has also been applied
to topics such as energy
landscapes in particle systems \cite{mirth2020representations}.
For other applications and references, see also
\cite{adams2021machine,beketayev2018measuring,Carlsson2014topological,carlsson2009topology,carlsson2019homotopy,edelsbrunner2008survey,munch2020hurricanes}.

We construct a filtered complex $(X_\lm,f_{\lm})$ associated to a subset
$X\subset \R^d$, a function $f:X\rightarrow \R$
which is bounded below, 
and a set of landmark points $\lm=\{\lmp_1,...,\lmp_n\} \subset X$. 
We have a simplex 
$\Delta=\{i_0,...,i_k\} \in X_\lm$ 
if certain higher order Voronoi cells 
corresponding to the $(k+1)!$ ways of ordering 
the elements of $\Delta$ are nonempty,
and $f_\lm(\Delta)$ is determined by the
minimum values of $f$ over those regions,
which may be practically obtained using optimization software. 
We consider the case of $X\subset \R^d$,
though this construction may be readily extended to a 
metric space with a locally convex metric, 
or to a space which is equipped with a 
map $X\rightarrow \R^d$, which may or may not be an inclusion map.

In Section \ref{sec:struct}, we review basic definitions
and construct the filtered complex $(X_\lm,f_\lm)$.
In Section \ref{sec:thm}
we prove that under certain conditions on $f$ such as being the
restriction of a smooth function, 
we have that $H_k^{a,b}(X_L) \cong H_k^{a,b}(X)$
for generic $a<b$, provided that $\lm$ is sufficiently dense in $X(b)$. In Section \ref{sec:ex}, we compute the persistent homology
groups in examples involving data sets, 
a higher-dimensional one involving the 
continuous form of the Ising model from statistical mechanics
\cite{nishikawa1976ising}, and
a function $f$ whose extremal sets
are the configuration space of 3 distinct points in $\R^2$.

\section{Construction of the complex}

\label{sec:struct}

We begin with some preliminary definitions,
and then define our main construction.

\subsection{Sublevel set persistent homology}
A simplicial complex $K$ will mean an abstract simplicial complex, which is a
collection of nonempty subsets $\Delta\subset S$ of some 
set of vertices $S$,
which is closed under taking nonempty subsets.
  A filtered simplicial complex is a pair $(K,f)$ where
  $K$ is a simplicial complex, and $f$ is a real-valued function
  $f : K\rightarrow \mathbb{R}$ on the simplices of $K$, such that
  $f(\sigma)\leq f(\tau)$ whenever $\sigma$ is a face
  of $\tau$.
Then for each $a\in \R$, we have that the sublevel set
$K(a)=f^{-1}(-\infty,a]$ is a subcomplex of $K$.
The persistent homology group is $H_k^{a,b}(K)$ is defined as the image of
\begin{equation}
  \label{phdef}
i^{a,b}_k : H_k(K(a))\rightarrow H_k(K(b)),
\end{equation}
where $i^{a,b} : K(a)\rightarrow K(b)$ is the inclusion
map for $a\leq b$. The Betti numbers are
denoted as usual by $\beta_k$. If $f:X\rightarrow \R$,
the sublevel set persistence $H^{a,b}_k(X)$ is
defined the same way as in \eqref{phdef}, when $X(a)$ is 
just the sublevel set $f^{-1}(-\infty,a]$ with the induced topology.

\subsection{Voronoi diagrams}
Let $\lm=\{\lmp_1,...,\lmp_n\} \subset \mathbb{R}^d$ be a subset.
\begin{defn}
\label{def:vor}For each nonempty subset $\{i_1,...,i_m\}\subset 
[n]=\{1,...,n\}$, we have the higher order Voronoi cell
\[\vor_{\{i_1,...,i_m\}}(\lm)=\left\{p : j \notin [n]\Rightarrow d(p,p_i)\leq d(p,p_j)\right\},\]
If $(i_1,...,i_m)$ are ordered, we also have the
\emph{ordered} Voronoi cell
\[\vor_{(i_1,...,i_m)}(\lm)=\left\{p \in \vor_{\{i_1,...,i_m\}}(\lm) :
    d(p,p_{i_1})\leq \cdots \leq d(p,p_{i_m})\right\}\]
\end{defn}
Notice that the above cells are all convex,
and that $\vor_i(\lm)=\vor_{\{i\}}(\lm)$ are the usual 
cells in the Voronoi diagram.
If $\Delta=\{i_1,...,i_m\}\subset [n]$ is a subset, we will
    denote its different orderings by
    $\ordering{\Delta}{\sigma}=(i_{\sigma_1},...,i_{\sigma_m})$,
when $i_1<\cdots <i_m$ are in sorted order, and $\sigma\in S_m$.

\subsection{The filtered complex}

Suppose $X\subset \R^d$, $f:X\rightarrow \R$ is
function which is bounded below, and
    let $\lm=\{\lmp_1,...,\lmp_n\}\subset \R^d$ be a collection of points, 
 called landmark points.
  \begin{defn}
  \label{def:main}
    Given the above data, 
define a filtered complex $(X_\lm,f_\lm)$ as follows. 
\begin{enumerate}
\item \label{xdeldef} We have a $k$-simplex $\Delta\in X_\lm$ if
 the ordered Voronoi diagram $\vor_{\ordering{\Delta}{\sig}}(\lm)\cap X$
 is nonempty for every $\sigma \in S_{k+1}$.
\item \label{xfiltdef} The filtration 
function is given by
  \begin{equation}
\label{filtx}
    f_\lm(\Delta)=\max_{\sig \in S_{k+1}}\ 
    \inf_{p\in \vor_{\ordering{\Delta}{\sig}}(\lm)\cap X} f(p).
  \end{equation}
\end{enumerate}
\end{defn}
Then $f_\lm(\Delta)$ is
well-defined since the infimum is over 
a nonempty set, and $f$ is bounded below.
It is clear from the obvious containment of
ordered Voronoi diagrams that $(X_\lm,f_\lm)$ 
satisfies the axioms of a filtered
complex. Notice that the position 
of landmark points $\lmp_j\notin X$
may affect $(X_L,f_L)$.

  \section{A convergence result}
  \label{sec:thm}

Given a subset $A\subset \R^d$, let
\[B_r(A)=\left\{p\in \R^d:
    d(p,A)<r\right\}=\bigcup_{p\in A} B_r(p)\]
be the union of the balls of radius $r$ 
centered at the points of $A$.
For instance, $A\subset B_r(\lm)$ if $\lm$ 
is an $r$-covering of $A$.

We will prove the following:
\begin{thm}
  \label{thm:main}
  Suppose $X=\R^d$, and $f\in C^\infty(\R^d)$
  is smooth with compact sublevel sets $f^{-1} (-\infty,a]$, and $a<b$ are not critical values.
  Then there exists $\delta>0$ such that
  \begin{equation}
    \label{mainthmeq}
    H_k^{a,b}(X_\lm)\cong H_k^{a,b}(X)
    \end{equation}
  whenever $\lm\in\R^d$ satisfies $f^{-1}(-\infty,b]\subset B_\delta(\lm)$.
\end{thm}

\begin{remark}
By Sard's theorem the 
set of points $a<b$ where the
theorem does not hold are of measure zero,
so the theorem may be interpreted as a 
sort of convergence result of persistence diagrams.
It would be interesting to additionally
understand the convergence rate, in other words the 
dependence of $\delta$ on $\epsilon$.
Additionally, while the examples we consider
in Section \ref{sec:ex} are smooth or at
least continuous, in general one would like
to apply this construction
to discontinuous functions. It would be desirable 
to have a more general result that applies in this case.
\end{remark}

\begin{remark}
The non-shaded region $X\subset \R^2$ in Figure 
\ref{fig:bad}
illustrates what can go wrong if 
we allow $a=b$, in other words
take usual non-persistent homology.
  The region in the figure is contractible, but the 
  above arrangement will have a nontrivial
  1-cycle surrounding the triangle $(\lmp_2,\lmp_3,\lmp_4)$,
  so that $H^{0,0}_1(X_\lm)\ncong H^{0,0}_1(X)$. 
This diagram can happen at any scale $\delta>0$,
contradicting the conclusion of the theorem.
\end{remark}

\begin{figure}
\begin{center}
\begin{scaletikzpicturetowidth}{\textwidth/2}
\begin{tikzpicture}[scale=\tikzscale]
  \path[fill=gray] (-10,-7)--(10,-7)--(10,-.75)--(-10,-3.15)--(-10,-7);
  \path[fill=darkgray]
  (-4,-4)--(0,-6)--(0,-4)--(-4,-4);
  \draw[fill] (0,8) circle [radius=0.2];
  \node[above] at (0,8) {$\lmp_1$};
    \draw[fill] (0,-16) circle [radius=0.2];
    \node[below] at (0,-16) {$\lmp_4$};
        \draw[fill] (-8,0) circle [radius=0.2];
        \node[left] at (-8,0) {$\lmp_2$};
                \draw[fill] (8,0) circle [radius=0.2];
  \node[right] at (8,0) {$\lmp_3$};
  \draw[thick,dashed] (0,-6) to (0,0);
  \draw[thick,dashed] (0,-6) to (4,-8);
  \draw[thick,dashed] (0,-6) to (-4,-8);
  \draw[thick] (-8,0) to (8,0);
  \draw[thick,dashed] (0,0) to (-4,4);
    \draw[thick,dashed] (0,0) to (4,4);
  \draw[thick] (0,8) to (-8,0) to (0,-16) to (8,0) to (0,8);
  \end{tikzpicture}
\end{scaletikzpicturetowidth}
\end{center}
\caption{
Let $X=\R^2$ and $f(x,y)=-y(8y-x)$, 
so that $X(0)=f^{-1}(-\infty,0]$ 
is the complement of the gray region,
shown above in a particular range.
  Four landmark
  points are shown.
  The region $\vor_{(2,3,4)}(\lm)$ is shown in dark
  gray, which is contained in the complement of $X(0)$,
  whereas each of the Voronoi regions representing
  the three boundary segments intersect $X(0)$ nontrivially,
  leading to a non-boundary one-cycle.
  We then have that
  $H_1^{0,0}(X)\ncong H_1^{0,0}(X_\lm)$, at any scaling
  factor $\lm=\{\delta\lmp_1,\delta\lmp_2,\delta\lmp_3,\delta\lmp_4\}$.}
    \label{fig:bad}
\end{figure}
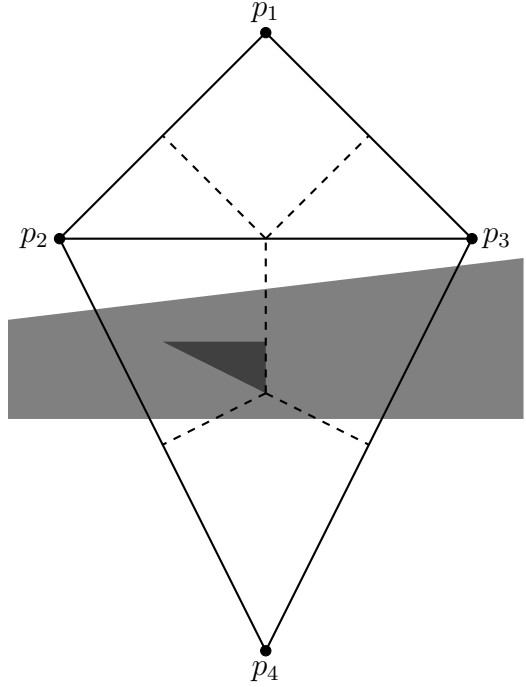

Let
\begin{equation}
  \label{eq:delaunay}
  D_\lm=\left\{\Delta \subset [n]:
    \bigcap_{i \in \Delta} \vor_{i}(\lm) \neq \emptyset\right\}
  \end{equation}
    be the Delaunay complex, 
    which is the nerve of
    the usual Voronoi diagram $\vor(\lm)=\{\vor_i(\lm)\}$.
    In other words, $\Delta \in D_\lm$ if there exists
    a point $p\in \R^d$, and $c\geq 0$ for which 
    \begin{equation}\label{eq:deldist}
    \mbox{$d(p,p_i)=c$ for $i\in \Delta$}, \quad
    \mbox{$c\leq d(p,p_j)$ for all $j\in [n]$}.
    \end{equation}
If the $\lm$ are in general position, 
  we obtain the the Delaunay
  triangulation of $\R^d$.

\begin{lemma}
     \label{lem:vordel}
     If $X=\R^d$, then we have that $X_\lm= D_\lm$.     
\end{lemma}

\begin{proof}

It is clear that $D_\lm \subset X_\lm$,
for if $p\in \vor_i(\lm)$ for all $i\in \Delta$,
then $p\in \vor_{\ordering{\Delta}{\sig}}(\lm)$
for all $\sigma\in S_{k+1}$. For the reverse
inclusion, suppose that
$\vor_{\ordering{\Delta}{\sig}}(\lm)$ 
is nonempty for all $\sigma\in S_{k+1}$.
We must show that there is a point
$p$ satisfying \eqref{eq:deldist}.

Let $\vor_{\Delta}(\lm)=B_0\cup B_1$ 
where $B_0$ are those points $p\in \vor_{\Delta}(\lm)$ 
which are in the half space
$d(p,\lmp_{i_0})\leq d(p,\lmp_{i_1})$,
and similarly for $B_1$, where $\vor_\Delta(\lm)$ is the non-ordered
higher Voronoi cell.
By assumption we have that 
$\vor_{(i_{\tau_1},...,i_{\tau_k})}(\lm)$ intersects both $B_0,B_1$
nontrivially for any $\tau\in S_k$, 
so by induction on $k$, we have a point $q_j \in B_j$
for which $d(q_j,\lmp_1)=\cdots =d(q_j,\lmp_{k})$ for $j=0,1$. Then the line segment connecting
$q_0$ to $q_1$ crosses the hyperplane
$B_0\cap B_1$, 
and the point $p$ of intersection satisfies
\eqref{eq:deldist}.
\end{proof}

Let $|K|$ be the geometric realization of $K$, which is the set
of maps $t:[n]\rightarrow [0,1]$ satisfying 
\begin{equation}
    \label{eq:geometricdef}
    \{i\in [n]:t_i>0\}\in K,\quad 
    \sum_{i\in [n]} t_i=1.
\end{equation}
Define $\pi:|X_{\lm}|\rightarrow \R^d$ by
\begin{equation}\label{eq:defpi} \pi(t_1,...,t_n)= t_1 \lmp_{1}+\cdots +t_n \lmp_{n}.
\end{equation}

\begin{lemma}
      \label{topylem}
  For any subset $A\subset \R^d$, we have that
  $\pi:|A_\lm|\rightarrow \R^d$
   is a homotopy equivalence onto its image.
\end{lemma}

\begin{proof}
First, for each $p\in \R^d$, we have a subset
\[\Delta_p=\left\{i\in [n]: p\in \vor_i(\lm)\right\}
\in D_{\lm},\]
which produces those cells for
which we can have inequality on the right hand
side of \eqref{eq:deldist}.
We have a subset 
$D'_\lm=\{\Delta_p:p\in \R^d\}\subset D_\lm$
which in general is not a complex,
but consists of all of $D_\lm$ if $\lm$ are 
in general position.
Then we have that
\[\R^d=\bigsqcup_{\Delta\in D'_\lm} U_\Delta,\quad
U_\Delta=\pi(|\Delta|)\]
is a partition of $\R^d$.

To prove the lemma, it suffices to show that
the inverse image of any point is contractible.
To see this, 
if $p\in U_{\Delta}$, we have that
$\pi^{-1}(\{p\})$ is contained entirely
in $|\Delta|$, and the inverse image is
an affine space which is contractible.
Since $A_\lm\subset D_\lm$ is a subcomplex,
that statement holds for the restriction of 
$\pi$ to $|A_\lm|$ as well, from which result follows.
\end{proof}

\begin{lemma}
\label{lem:inc}
Under the assumptions of
Theorem \ref{thm:main}, 
let $a\in \R$, $\eps>0$. Then
there exists $\delta>0$ so that
\begin{equation}
    \label{eq:inclem}
X(a-\eps) \subset 
\pi(|X_{\lm}(a)|) \subset X(a+\eps)
\end{equation}
whenever $\lm$ satisfies $X(a)\subset B_\delta(\lm)$.
\end{lemma}

\begin{proof}
Since the sublevel sets of $g$ are compact by the
assumptions of the theorem,
we have that $g$ is absolutely continuous 
on each one.
Then there exists $r,r'>0$ so that
\[B_{r}(X(a-\eps)) \subset X(a),
\quad B_{r'}(X(a))\subset X(a+\eps).\]

  For the first inclusion in \eqref{eq:inclem}, 
  choose $\delta=r/2$. Then 
  any $p\in X(a-\eps)$ is contained in
  some $U_{\Delta}\subset \mathbb{R}^d$,
  which must be entirely contained in $X(a)$.
Similarly, if $p\in U_\Delta$ is such
that all the vertices of $U_\Delta$ 
are in $X(a)$, then $p$ must
be within $2\delta$ of $X(a+\eps)$, so we may choose $\delta=r'/2$. 
\end{proof}

We can now prove the theorem.
\begin{proof}
  By the Theorem 3.1 of 
  \cite{milnor1969morse},
  and the fact that the set of critical values is closed, there exists a neighborhood
  $[a-\eps,a+\eps]$ so that $X_{a+\eps}$ deformation retracts onto $X_{a-\eps}$,
  and similarly for $b$. It follows that
  \begin{equation}
    \label{proofeq1}
H_k^{a-\eps,b-\eps}(X)\rightarrow H_k^{a+\eps,b+\eps}(X)
    \end{equation}
is an isomorphism.

By Lemma \ref{lem:inc}, we can find $\delta>0$ so that
$H_k^{a,b}(X)\cong H_k^{a,b}(\pi(|X_\lm|))$ where the right
side is the image 
\[H_k(\pi(|X_\lm(a)|))\rightarrow H_k(\pi(|X_\lm(b)|)),\]
  since the isomorphism in \eqref{proofeq1} factors through
  both sides.
  By Lemma \ref{topylem},
  and the obvious compatibilities between these
 inclusion maps and $\pi$, we find that
  $H_k^{a,b}(X_\lm)\cong H_k^{a,b}(\pi(|X_\lm|))$. 
  Combining the
  two isomorphisms completes the proof.
  
\end{proof}

\section{Examples}
\label{sec:ex}
We illustrate the construction of Definition \ref{def:main} in some examples.
We have used the \javaplex 
package \cite{Javaplex} to generate all persistence
diagrams. In Sections \ref{subsec:datasets} and \ref{subsec:config-space}, we used the MATLAB function \texttt{fmincon} to obtain (local) minimizers over ordered Voronoi cells, per the infimum in equation (\ref{filtx}), with the \texttt{sqp} option enabled.  In Section \ref{subsec:ising}, we found global minimizers over ordered Voronoi cells using Gurobi \cite{gurobi}.

\subsection{Sublevel set persistence for data sets}

\label{subsec:datasets}

We begin with an application 
to data sets via the method discussed in the introduction.

Given a finite subset $D=\{z_1,...,z_N\}\subset \R^d$,
we define the following natural method for estimating density,
though others may be used as well. 
Select a real number $1<h<N$, and define
\[
\rho_i(z)=\exp(-\beta_i \lVert z-z_i\rVert^2),\]
where $\beta_i>0$ is defined to be the 
unique number with the property that 
\[\rho_i(z_1)+\cdots+\rho_i(z_N)=h.\]
We define a density estimator by
\begin{equation}
\nonumber
    \rho(z)=h^{-1}N^{-1}\sum_{i=1}^N \rho_i(z).
\end{equation}

We then choose $n$ landmark points $\lm \subset D$ either at random,
or using a greedy max of min distance algorithm so that they are as spread out
as possible. In other words, beginning with some randomly sampled points,
we consecutively build up $\lm$ by adding the point
whose minimum distance is as large as possible 
from the points that are already in the set. 
While randomly sampling is more objective, 
Theorem \ref{thm:main} indicates that the second method
may produce a desirable complex with fewer points.
We then let $X=\R^d$, and define the filtration 
function in the natural
way $f(z)=-\log(\rho(z))$, so that denser points have lower persistence values.

The resulting persistence diagrams are shown 
for a data set from
\cite{scikittda2019} in Figure \ref{fig:infinity}, 
which shows a noisy data set with 1000 points
roughly lying on an infinity symbol, 
with five dense regions at the center and in the corners. 
The density function was determined by the above method using $h=50$. 
The corresponding complex is shown with persistence values, 
as well as the first two persistence diagrams.
both the $\beta_0$ and $\beta_1$ features
are prominent in the persistence diagrams, which has minimal noise.

\begin{figure}
     \centering
     \begin{subfigure}[b]{0.45\textwidth}
         \centering
         \includegraphics[width=\textwidth]{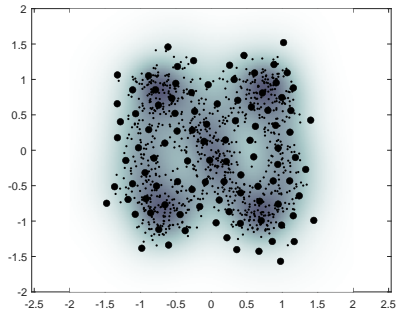}
         \caption{The function $f(x)$ and landmarks $\lm$
         on top of the data set $D$}
         \label{fig:inf1}
     \end{subfigure}
     \hfill
     \begin{subfigure}[b]{0.45\textwidth}
         \centering
         \includegraphics[width=\textwidth]{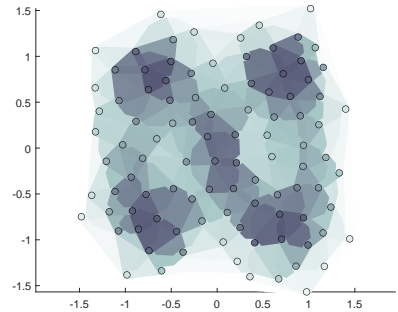}
         \caption{The complex $X_{\lm}$, with persistence values indicated by shading}
         \label{fig:inf2}
     \end{subfigure}
     \begin{subfigure}[b]{0.45\textwidth}
         \centering
         \includegraphics[width=\textwidth]{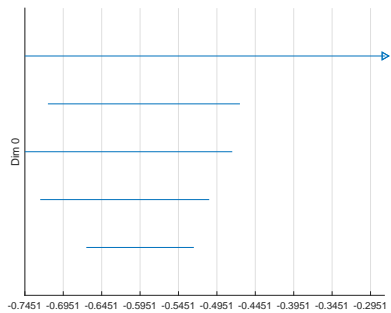}
         \caption{$\beta_0$ diagram}
         \label{fig:inf3}
     \end{subfigure}
          \begin{subfigure}[b]{0.45\textwidth}
         \centering
         \includegraphics[width=\textwidth]{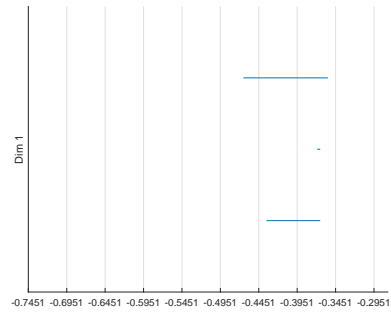}
         \caption{$\beta_1$ diagram}
         \label{fig:inf4}
     \end{subfigure}
        \caption{The ``infinity'' data set from \cite{scikittda2019}. The
        function $f(z)=-\log(\rho(z))$ and landmark points
        $\lm$ are overlayed on the data set
        in Figure \ref{fig:inf1}, and $X_\lm$ is
        shown in \ref{fig:inf2}. The 
        first two \javaplex
        persistence diagrams are shown in \ref{fig:inf3} and \ref{fig:inf4}.}
        \label{fig:infinity}
\end{figure}

\subsection{The continuous Ising model\label{subsec:ising}}

Our second example illustrates a higher-dimensional
persistence function which appears in statistical mechanics.

Let 
\[X=\left\{\sigma\in \R^d: |\sigma_i|\leq 1\right\},\]
be the state space of the continuous form of the one-dimensional Ising model
on $d$-sites, whose discrete form includes only the values $\sigma_i\in \{1,-1\}$.
The values $\sigma_i$ are called the ``spin'' values. Let
\[H(\sigma)=-\sum_{i=1}^{d-1} \sigma_i \sigma_{i+1}\]
be the corresponding Hamiltonian function with no external field.
We also consider the circular case, in which we add
a $-\sigma_n \sigma_1$ term, corresponding to the Ising model on the circle.
We define $f(\sigma)=H(\sigma)$ in either case. 
States with low values of $H(\sigma)$ tend to be ones
in which neighboring points are similar. 
There are two global minimizers when
all $\sigma_i$ are all simultaneously equal to $\pm 1$, in which case
we have $H(\sigma)=-(d-1)$. At higher energy levels, one expects
interesting homological features in the continuous
limit, as the number of sites becomes large.

We selected landmarks point by selecting 20 distinct
states with values $\sigma_i\in \{\pm 1\}$ among those with low
energy values, and we computed the corresponding persistent
homology groups $H_k^{a,b}(X_\lm)$ for both the interval and
and circle example. The persistence diagrams are shown in
Figure \ref{fig:ising}, together with illustrations of the
higher dimensional features.

\begin{figure}
     \centering
          \begin{subfigure}[b]{.49\textwidth}
    \centering
    \includegraphics[width=.49\textwidth]{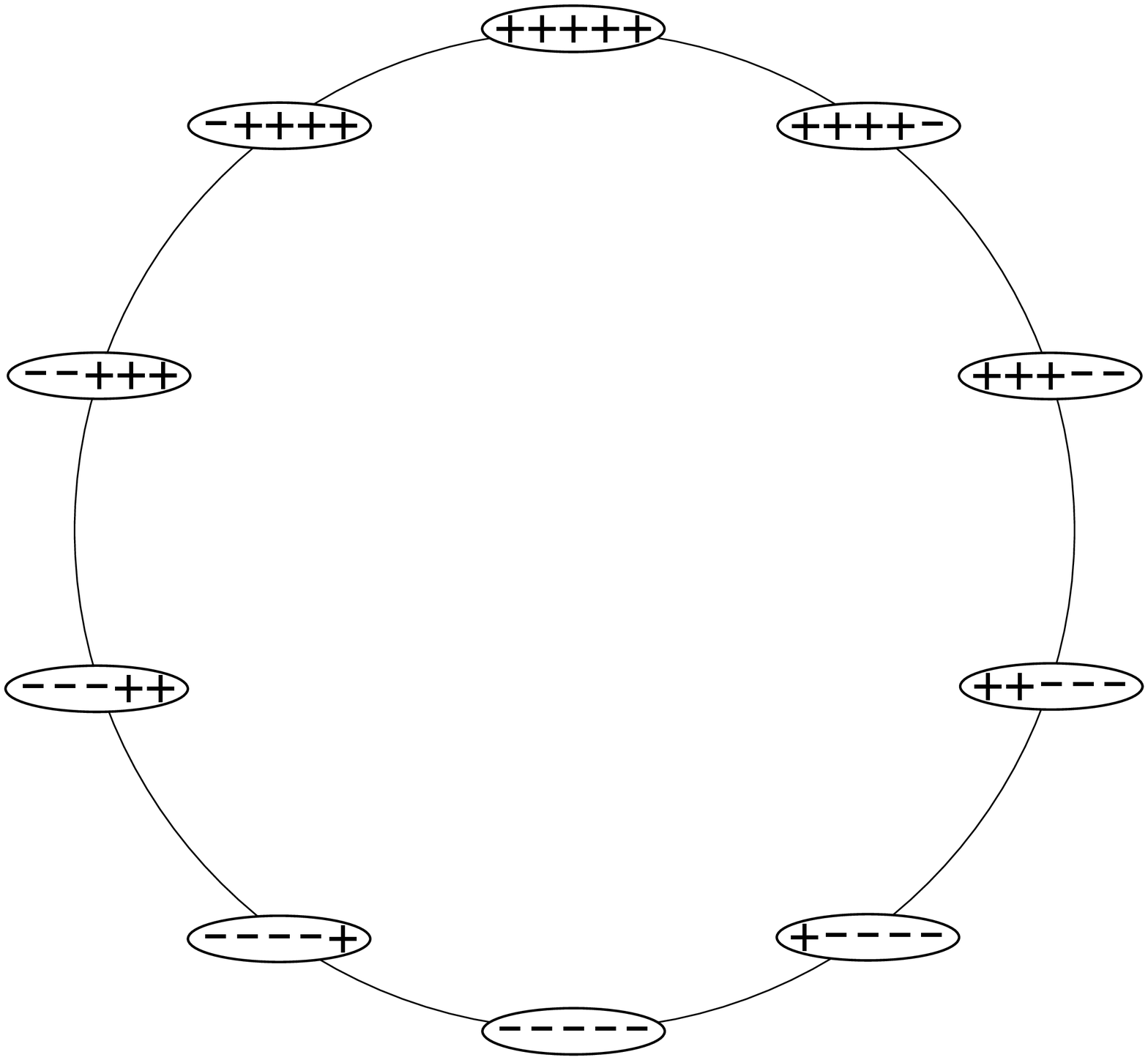}
    \caption{A 1-cycle for the interval case}
    \label{fig:ising-int}
\end{subfigure}
     \begin{subfigure}[b]{.49\textwidth}
    \centering
    \includegraphics[width=.49\textwidth]{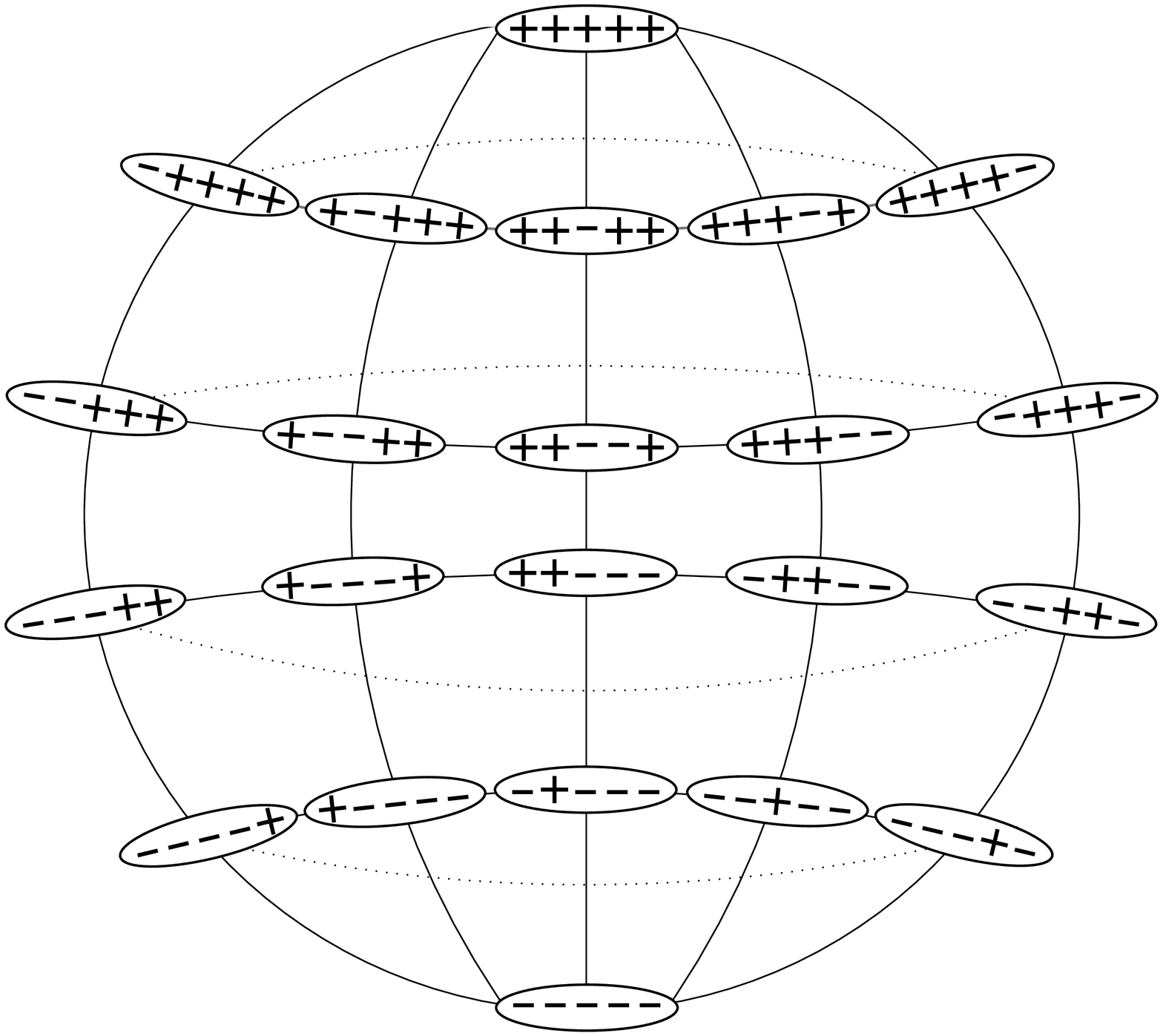}
    \caption{A 2-cycle parametrized by the sphere}
    \label{fig:ising-sphere}
\end{subfigure}
     \begin{subfigure}[b]{0.3\textwidth}
         \centering
         \includegraphics[width=\textwidth]{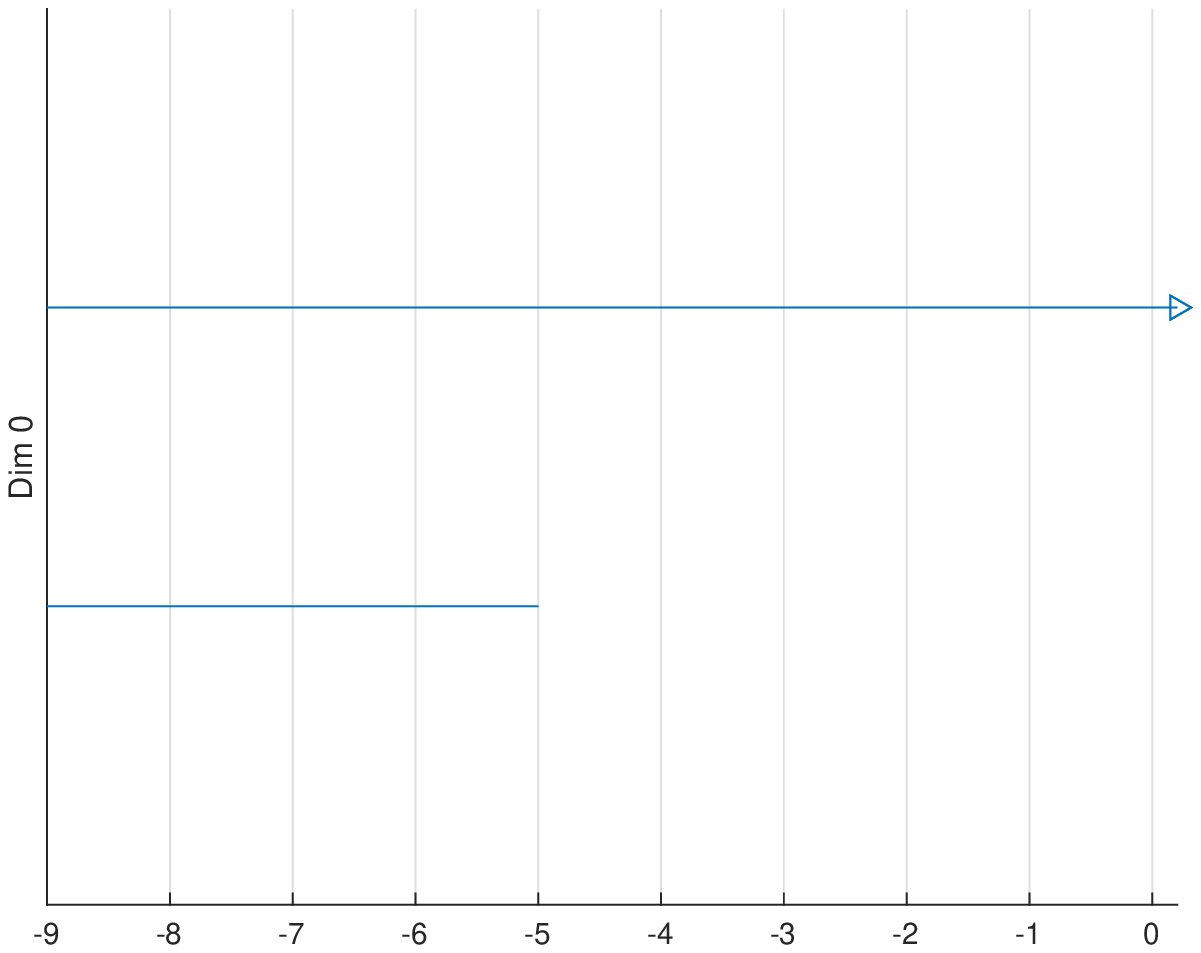}
         \caption{$\beta_0$ diagram, circle}
         \label{fig:ising-circ-betti0}
     \end{subfigure}
     \hfill
     \begin{subfigure}[b]{0.3\textwidth}
         \centering
         \includegraphics[width=\textwidth]{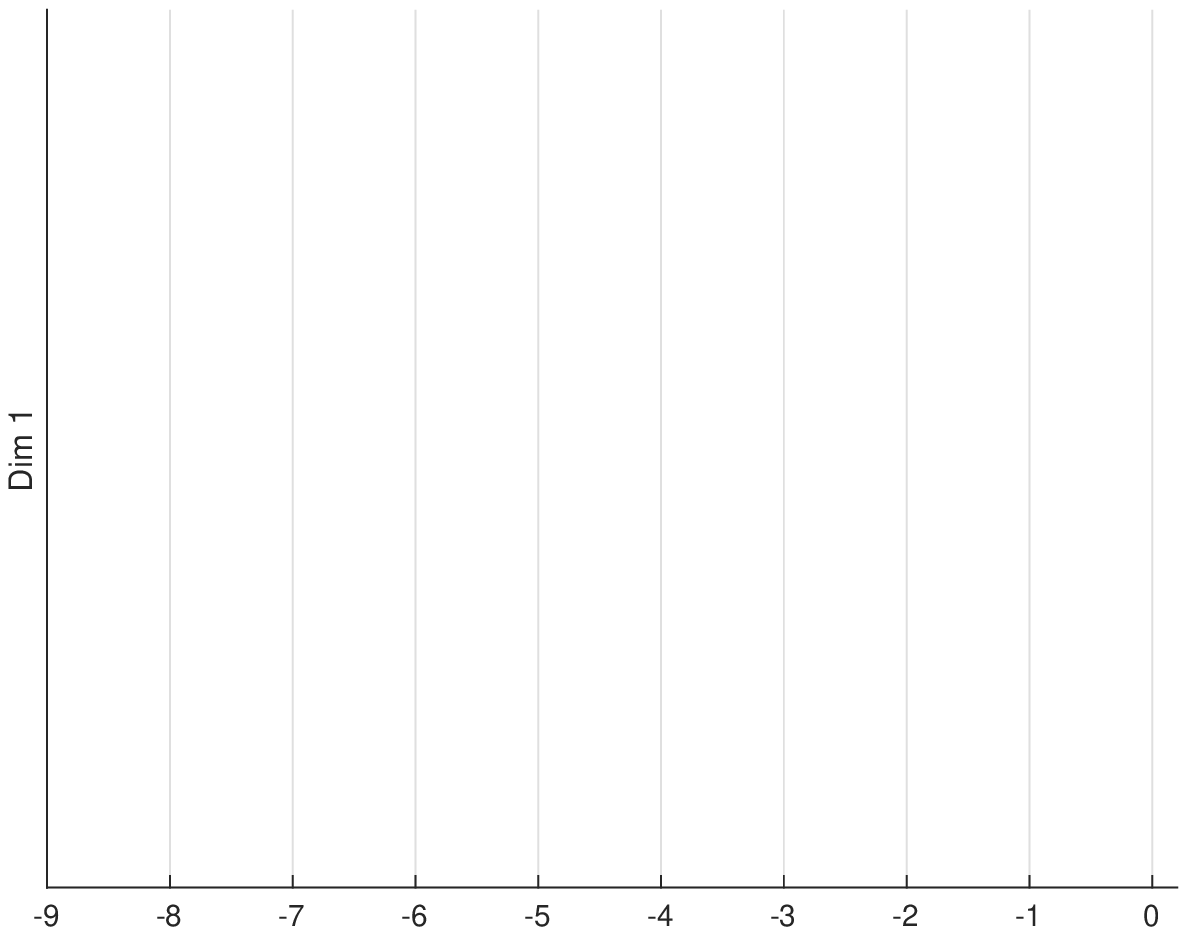}
         \caption{$\beta_1$ diagram, circle}
         \label{fig:ising-circ-betti1}
     \end{subfigure}
     \begin{subfigure}[b]{0.3\textwidth}
         \centering
         \includegraphics[width=\textwidth]{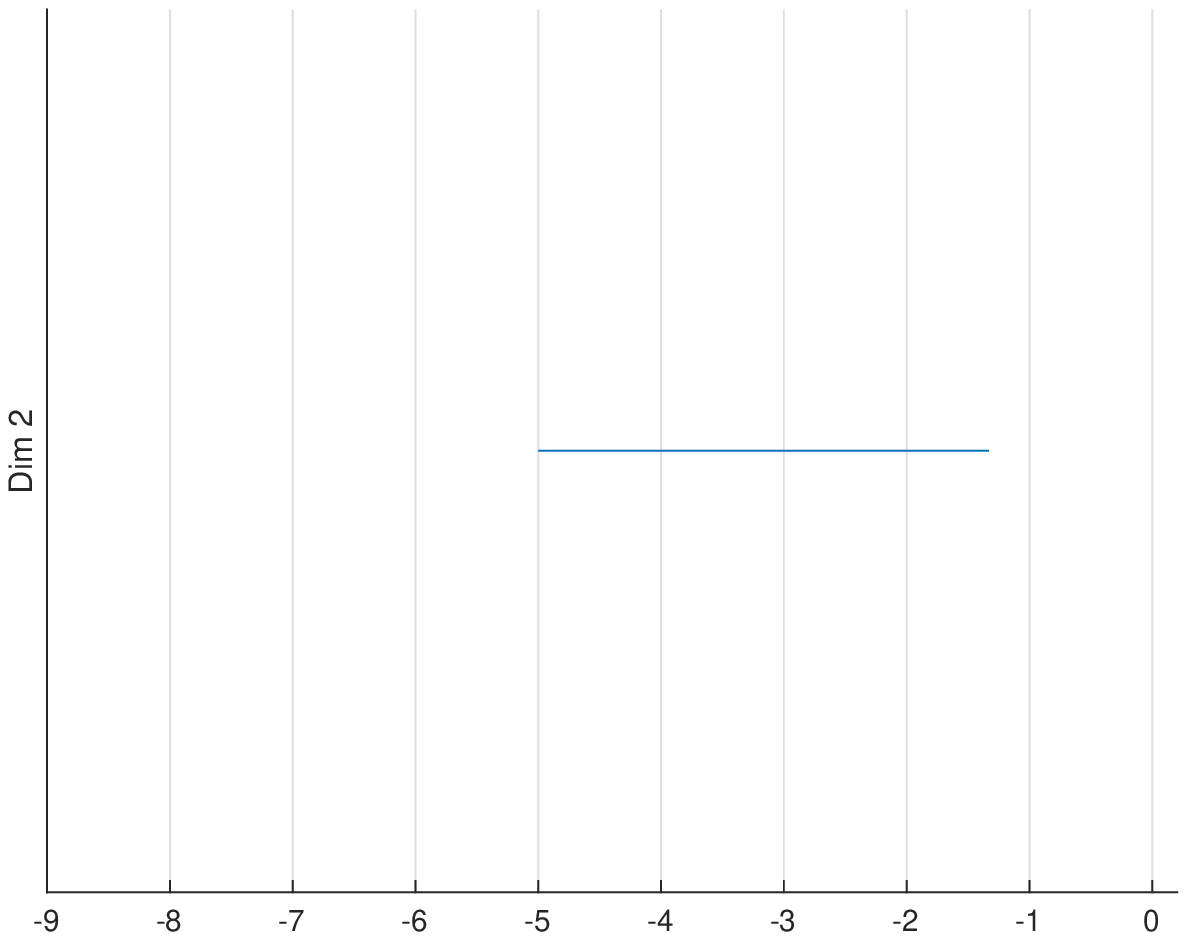}
\caption{$\beta_2$ diagram, circle}
         \label{fig:ising-circ-betti2}
     \end{subfigure}
          \begin{subfigure}[b]{0.3\textwidth}
         \centering
         \includegraphics[width=\textwidth]{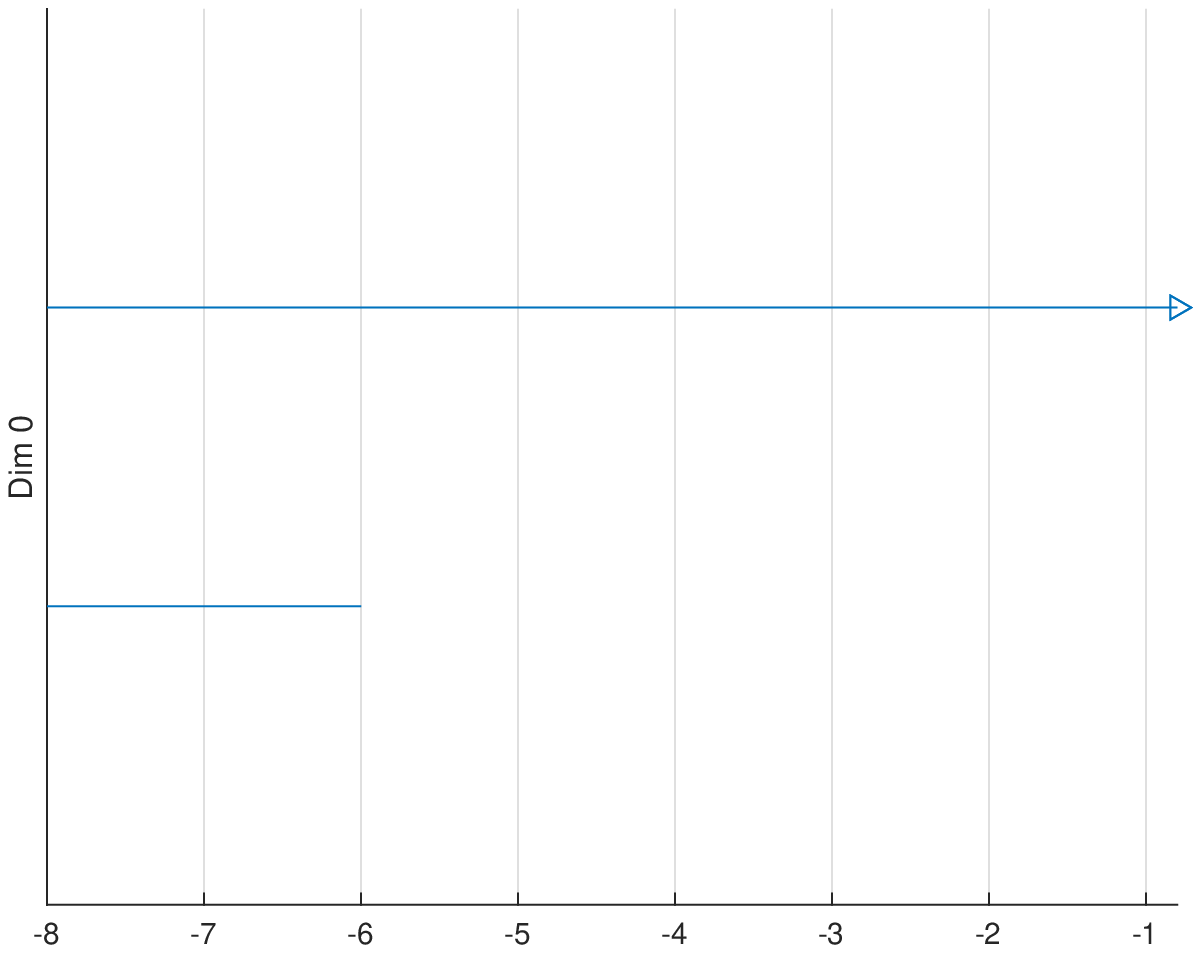}
\caption{$\beta_0$ diagram, interval}
         \label{fig:ising-int-betti0}
     \end{subfigure}
     \hfill
     \begin{subfigure}[b]{0.3\textwidth}
         \centering
         \includegraphics[width=\textwidth]{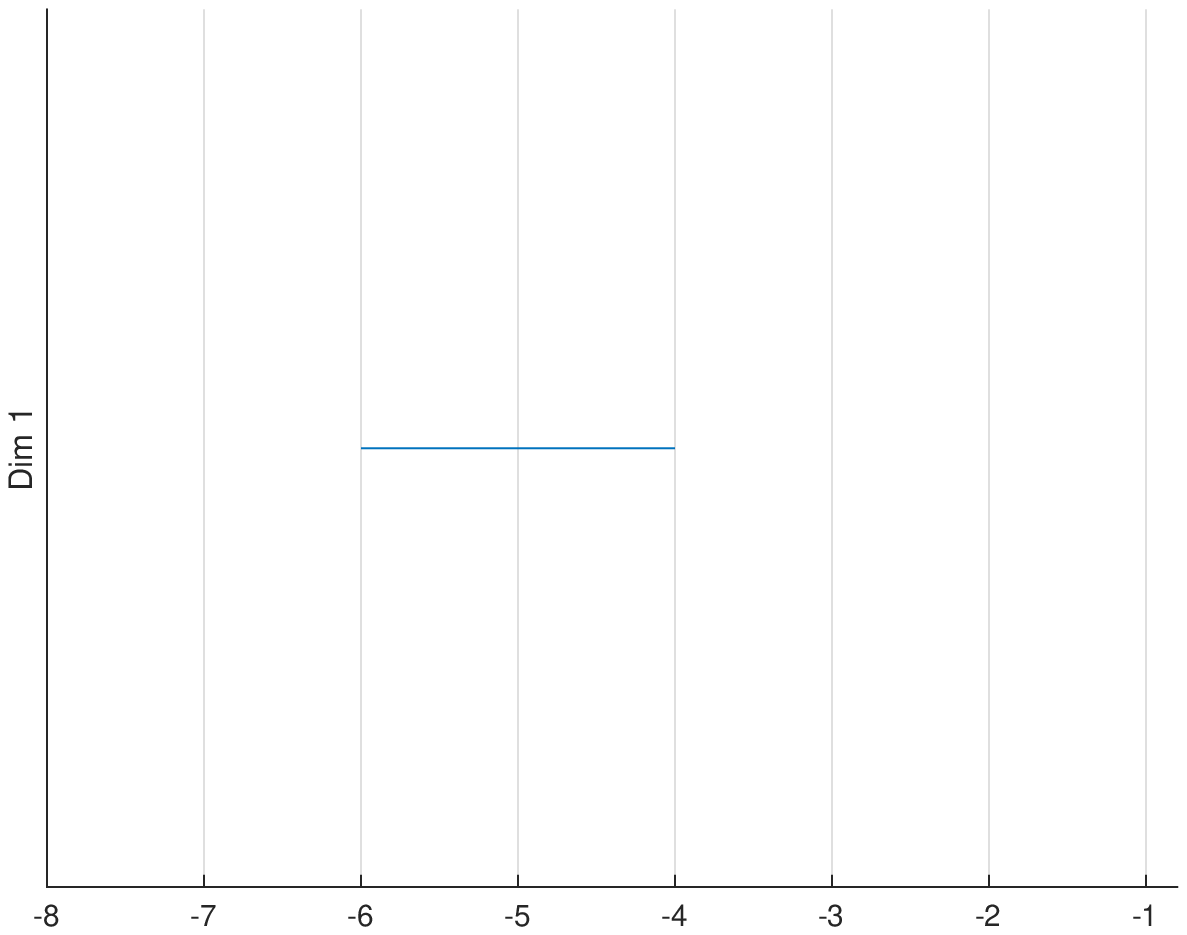}
\caption{$\beta_1$ diagram, interval}
         \label{fig:ising-int-betti1}
     \end{subfigure}
     \begin{subfigure}[b]{0.3\textwidth}
         \centering
         \includegraphics[width=\textwidth]{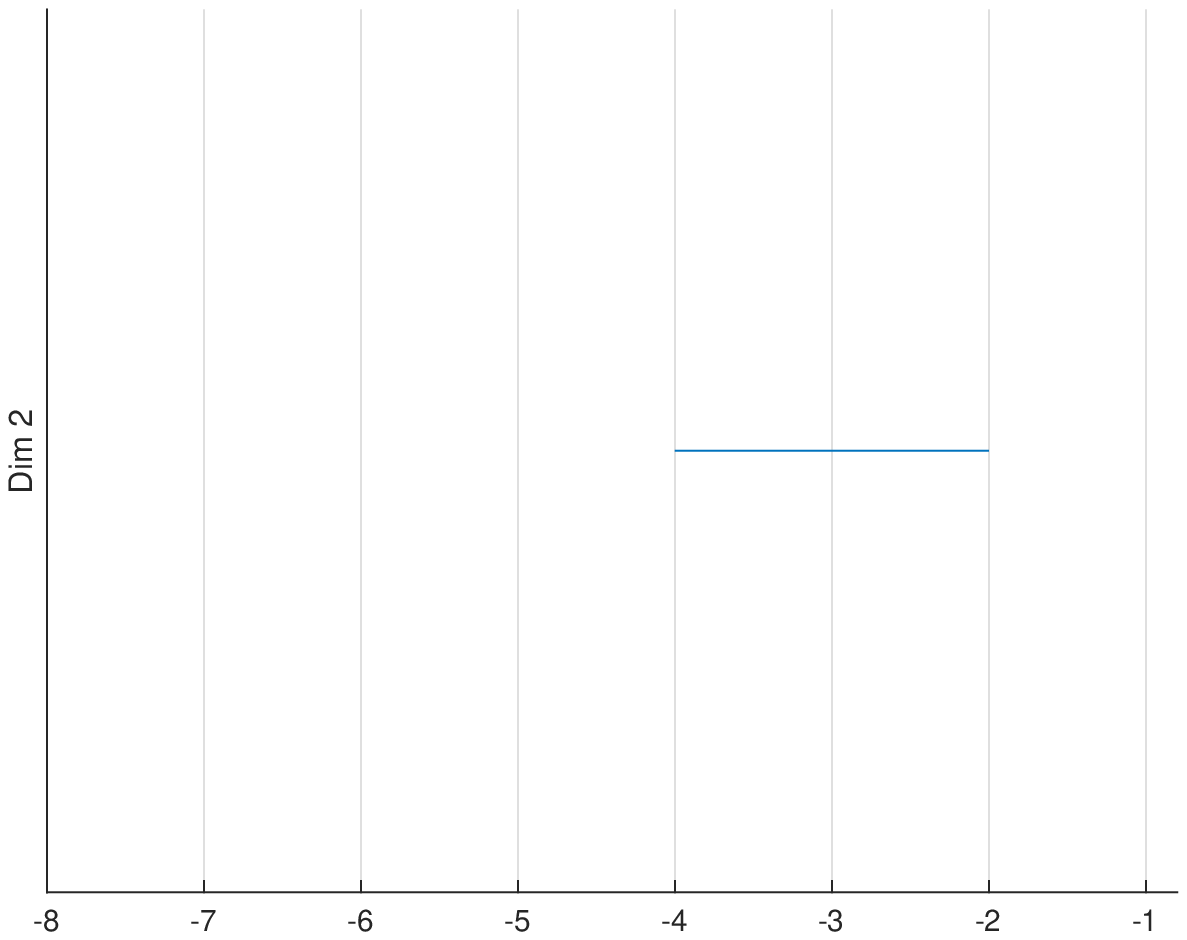}
\caption{$\beta_2$ diagram, interval}
         \label{fig:ising-int-betti2}
     \end{subfigure}
        \caption{The top row shows representative 1- and 2-cycles for the two instances of the Ising model considered in this section (the interval and a circle, respectively). The middle and bottom rows show the for 9 sites on an circle (middle) and the interval (bottom). Each one has two $\beta_0$
        features at lowest energy corresponding to all spin values being
        the same. The top row illustrates a nontrivial 1-cycle for the interval,
        and the nonzero $\beta_2$ feature that is present in both.}
        \label{fig:ising}
\end{figure}

\subsection{Configuration space\label{subsec:config-space}}

In the last example, we apply the complex to
produce the Betti numbers of a topologically
interesting space, by inventing a function whose local minimizers
are the space.

Let $C_n$ be the configuration space of $n$ distinct
ordered points 
$p_i\in \mathbb{R}^2$, in other words the
complement of the diagonal in $(\mathbb{R}^2)^n$.
Then $C_n$ deformation retracts onto 
an $(n-1)$-dimensional singular
subspace $C'_n\subset C_n$ which sometimes
appears in robotics, in which 
$(p_1,...,p_n)$ is mean-centered, and each 
$p_i$ is of distance exactly one
to its nearest neighbor.
For instance, a typical point in $C'_5$ would be

\hspace*{\fill} 
\begin{tikzpicture}
\draw[-,very thick,dashed] (0,0)--(.612,.79);
\draw[-,very thick,dashed] (.612,.79)--(1.34,1.48);
\draw[-,very thick,dashed] (1.34,1.48)--(.86,2.36);
\draw[-,very thick,dashed] (1.34,1.48)--(2.30,1.76);
\filldraw[black] (0,0) circle (2pt) node[anchor=east] {$p_2$};
\filldraw[black] (.612,.79) circle (2pt) node[anchor=east] {$p_4$};
\filldraw[black] (1.34,1.48) circle (2pt) node[anchor=east] {$p_3$};
\filldraw[black] (2.30,1.76) circle (2pt) node[anchor=south] {$p_1$};
\filldraw[black] (.86,2.36) circle (2pt) node[anchor=east] {$p_5$};
\end{tikzpicture}
\hspace{\fill}

\noindent
where we have drawn a  a dashed line between a point and its nearest neighbor.
The homology $H_k(C'_n)\cong H_k(C_n)$
space is well-known \cite{cohen1995configuration},
and for $n=3$ the Betti numbers are given by
$(\beta_0,\beta_1,\beta_2)=(1,3,2)$.



\begin{figure}
     \centering
     \begin{subfigure}[b]{0.3\textwidth}
         \centering
         \includegraphics[width=\textwidth]{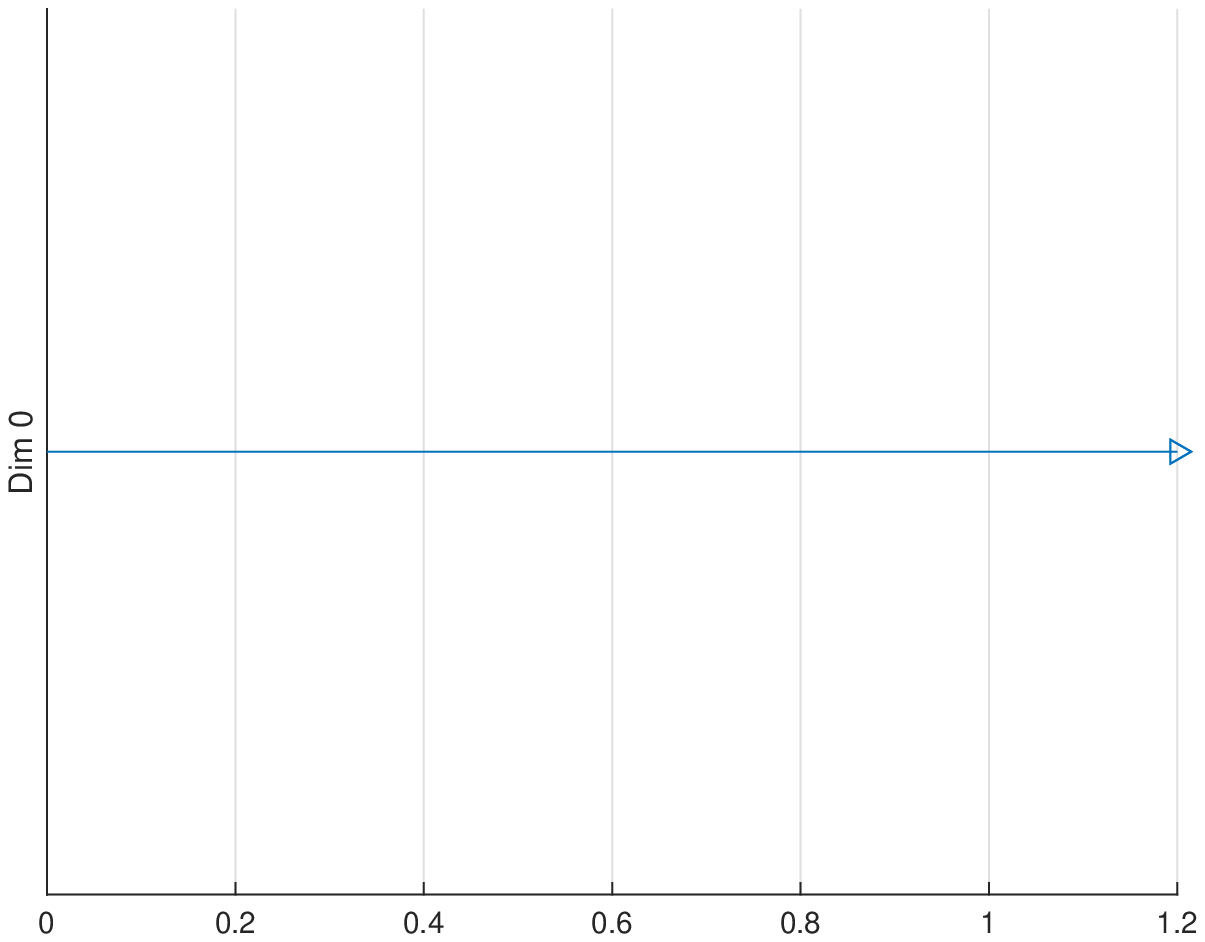}
         \caption{$\beta_0$ diagram}
         \label{fig:config1}
     \end{subfigure}
     \hfill
     \begin{subfigure}[b]{0.3\textwidth}
         \centering
         \includegraphics[width=\textwidth]{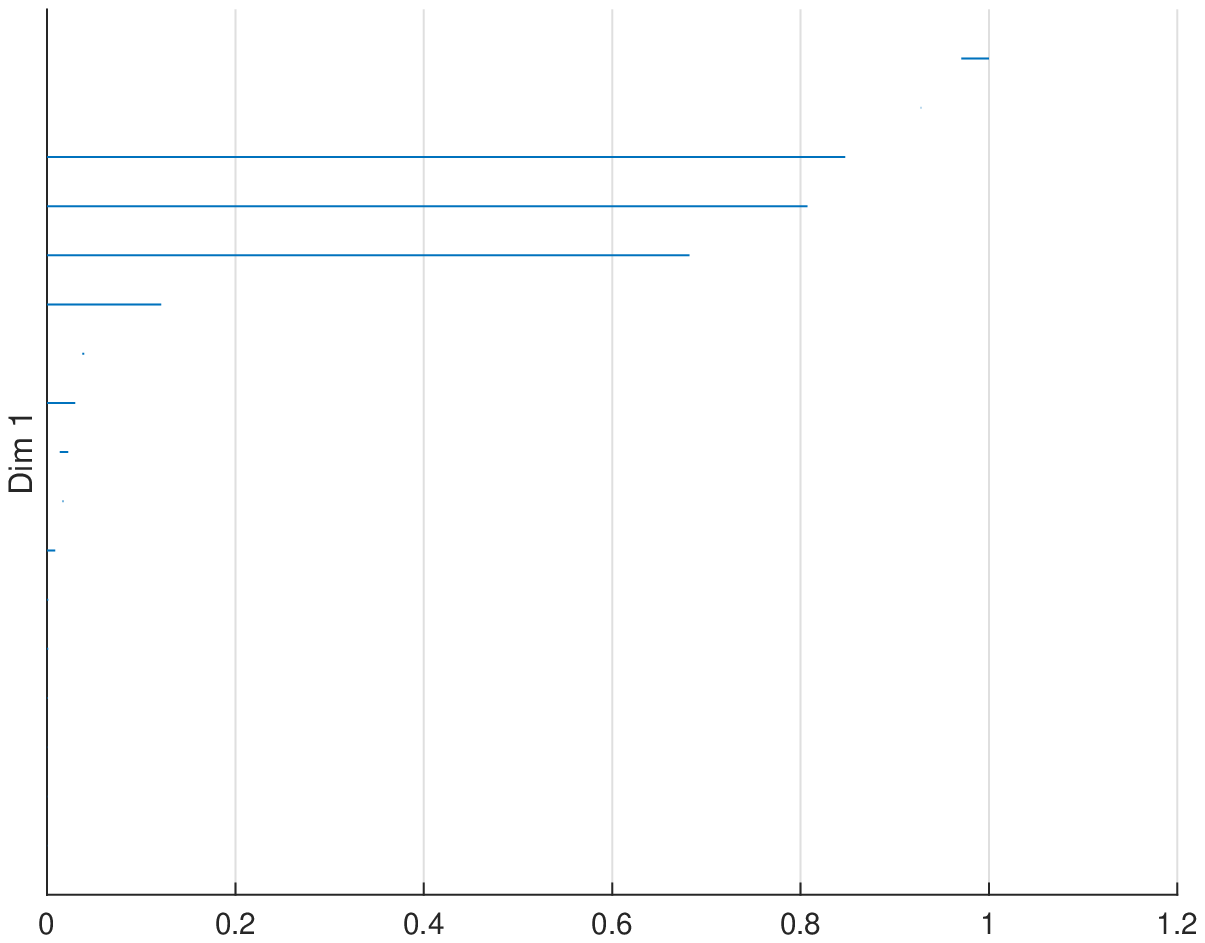}
\caption{$\beta_1$ diagram}
         \label{fig:config2}
     \end{subfigure}
     \hfill
     \begin{subfigure}[b]{0.3\textwidth}
         \centering
         \includegraphics[width=\textwidth]{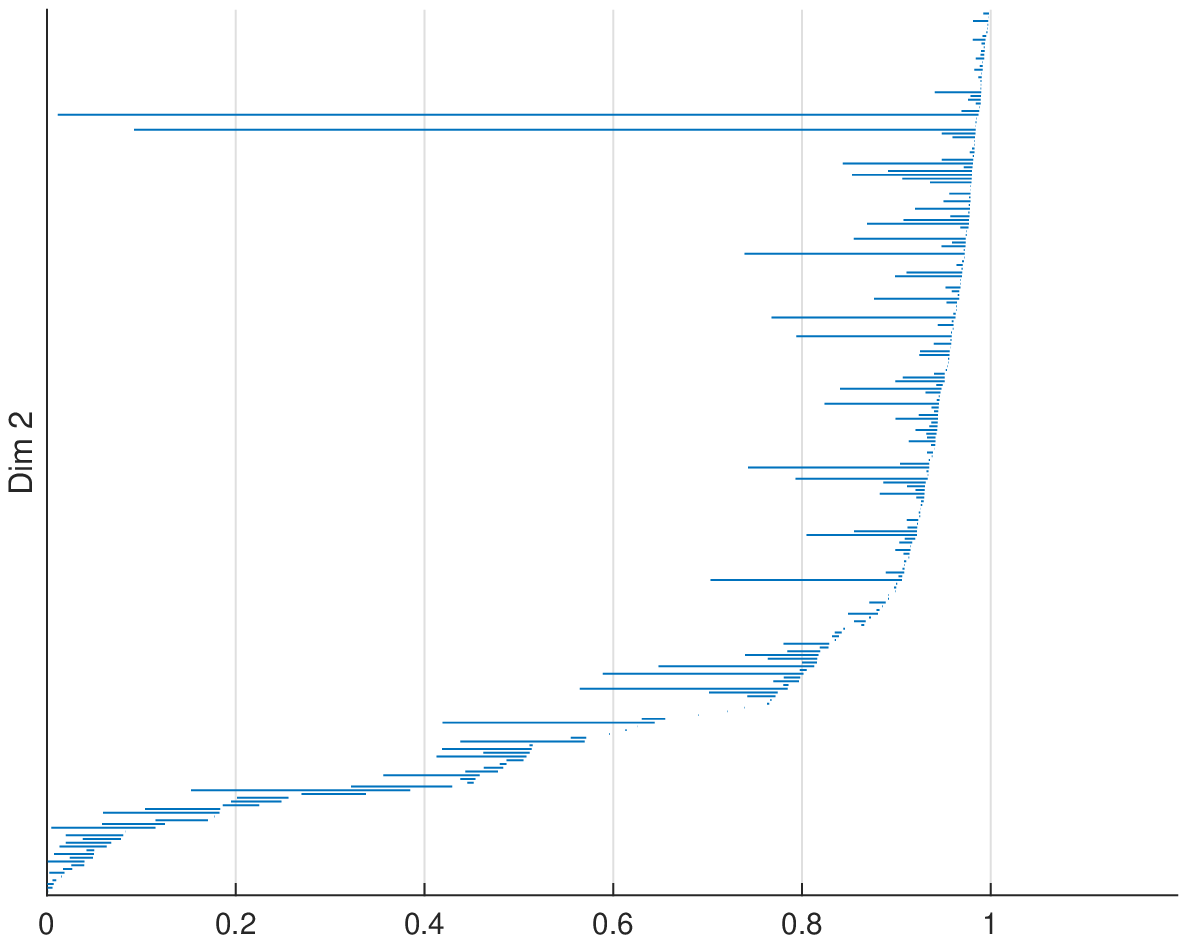}
\caption{$\beta_2$ diagram}
\end{subfigure}
     
        \caption{Persistence diagrams for the
        function in \eqref{eq:configfunc}
        whose extremal
        points are the configuration space $C'_3$.
        The three diagrams show the anticipated Betti numbers for this space,
        ending a little below the value $f(p_1,p_2,p_3)=1$, which is the smallest
        value at which points may collide. 
        }
        \label{fig:config}
\end{figure}

We then consider the following function, which has local minimizers
at the points of $C'_3$:
\begin{equation}
    \label{eq:configfunc}
f(p_1,p_2,p_3)=\max_i \left(\min_{j\neq i} \left(\lVert p_i-p_j \rVert^2-1\right)^2\right)
\end{equation}
We then take
\[X=\{(p_1,p_2,p_3)\in \R^6:p_1+p_2+p_3=(0,0)\},\] so that the sublevel
sets of $f$ are compact when taken as
a function on $X=\R^4$. We sampled
100 landmark points which are maximally spread out
in the range $f^{-1}(-\infty,.1]$ using the max of min distance
algorithm. The persistence diagrams
are shown in Figure \ref{fig:config}.

\bibliographystyle{plain}
\bibliography{refs}
\end{document}